\documentclass{amsart}

\usepackage[dvips]{graphicx}
\usepackage{xcolor}

\theoremstyle{plain}
\newtheorem{theorem}{Theorem}[section]
\newtheorem*{theorem*}{Theorem}
\newtheorem{lemma}[theorem]{Lemma}
\newtheorem{corollary}[theorem]{Corollary}

\theoremstyle{definition}
\newtheorem{definition}[theorem]{Definition}

\theoremstyle{remark}
\newtheorem{remark}[theorem]{Remark}
\newtheorem*{acknowledgements}{Acknowledgements}
\newtheorem*{Data}{Data Availability}

\newcommand{\del}{\partial}
\newcommand{\R}{\mathbb{R}}
\renewcommand{\H}{\mathbb{H}}
\renewcommand{\S}{\mathbb{S}}
\newcommand{\E}{\mathbb{E}}
\newcommand{\diam}{\mbox{diam}}

\begin{document}

\title{An upper bound on Pachner moves relating geometric triangulations}

\author[Kalelkar]{Tejas Kalelkar}
\address{Mathematics Department, Indian Institute of Science Education and Research, Pune 411008, India}
\email{tejas@iiserpune.ac.in}

\author[Phanse]{Advait Phanse}
\address{Mathematics Department, Indian Institute of Science Education and Research, Pune 411008, India}
\email{advait.phanse@students.iiserpune.ac.in}

\date{\today}

\keywords{Hauptvermutung, geometric triangulation, Pachner moves, combinatorial topology}

\subjclass[2010]{Primary 57Q25, 57R05}

\begin{abstract}
We show that any two geometric triangulations of a closed hyperbolic, spherical or Euclidean manifold are related by a sequence of Pachner moves and barycentric subdivisions of bounded length. This  bound is in terms of the dimension of the manifold, the number of top dimensional simplexes and bound on the lengths of edges of the triangulation. This leads to an algorithm to check from the combinatorics of the triangulation and bounds on lengths of edges, if two geometrically triangulated closed hyperbolic or low dimensional spherical manifolds are isometric or not. 
\end{abstract}

\maketitle

\section{Introduction}\label{intro}
The problem of determining if two given manifolds are homeomorphic has been extensively studied. Using ideas from Perelman's proof of the geometrization of closed irreducible 3 dimensional manifolds, Scott and Short\cite{ScoSho} have built on work by Manning, Jaco, Oertel and others to give an algorithm for the homeomorphism problem of such manifolds. More recently, Kuperberg\cite{Kup} has given a self-contained proof using only the statement of geometrization to show that the homeomorphism problem for 3-manifolds has computational complexity that is bounded by a bounded tower of exponentials in the number of tetrahedra.

Pachner\cite{Pac} has shown that any two simplicial triangulations of a manifold which have a common subdivision are related by a finite sequence of local combinatorial transformations called bistellar or Pachner moves. A bound on the number of such moves required to go from one triangulation of an $n$-manifold to another gives an algorithm to solve the homeomorphism problem for PL $n$-manifolds. Mijatovic in a series of papers gives such a bound for a large class of 3-manifolds \cite{Mij1} \cite{Mij2} \cite{Mij3} \cite{Mij4}. The bounds he obtains are also in terms of bounded towers of exponentials on the number of tetrahedra. In 1958, Markov\cite{Mar} had shown that the homeomorphism problem is unsolvable for manifolds of dimension greater than 3. This curtailed the search for a general algorithm applicable to manifolds of all dimension. For closed hyperbolic manifolds, the fundamental group is a complete invariant but it is not easy to algorithmically check if two Kleinian groups are isomorphic.

In this paper, we give an algorithmic solution for the homeomorphism problem on the restricted class of geometrically triangulated constant curvature manifolds, by obtaining a bound on the number of barycentric subdivisions and Pachner moves needed to relate them.

The \emph{geometric triangulation} of a Riemannian manifold is a finite simplicial triangulation where the interior of every simplex is a totally geodesic disk. Every Euclidean, spherical or hyperbolic manifold has a geometric triangulation. To see this, proceed as in Theorem 7.3 of \cite{HodRubSegTil}: Take the Dirichlet domain $\mathcal{D}$ of $M$ in its universal cover $\widetilde{M}$. As $\widetilde{M}$ is $\E^n$, $\H^n$ or $S^n$ so $\mathcal{D}$ is a convex geometric polyhedron. The boundary of $\mathcal{D}$ projects to the cut locus of a point in $M$, which gives a spine for $M$. We can now subdivide this cell decomposition into a geometric triangulation. 

In our proof we need the links of positive dimensional simplexes to be shellable spheres. We show that after sufficiently many barycentric subdivisions the links of simplexes do become shellable. As triangulated spheres of dimension at most 2 are always shellable, so for manifolds of dimension $n\leq4$ the links of positive dimensional simplexes are shellable and we do not need to take the initial barycentric subdivisions. The main result we prove in this paper is the following:

\begin{theorem}\label{mainthm}
Let $M$ be closed spherical, Euclidean or hyperbolic $n$-manifold with geometric triangulations $K_1$ and $K_2$. Let $K_1$ and $K_2$ have $p$ and $q$ many $n$-simplexes respectively with lengths of edges bounded above by $\Lambda$.When $M$ is spherical, we require $\Lambda \leq \pi/2$.  Let $inj(M)$ denote the injectivity radius of $M$. 

When $n\leq 4$, then $K_1$ and $K_2$ are related by $f(n, p, q, \Lambda,inj(M))$ many Pachner moves which do not remove common vertices. In general, their $2^{n+1}$-th barycentric subdivisions, $\beta^{2^{n+1}} K_1$ and $\beta^{2^{n+1}} K_2$ are related by $f(n, p, q, \Lambda, inj(M))$ many Pachner moves which do not remove common vertices. The bounding function $f$ is the following: 
$$f(n, p, q, \Lambda, inj(M))= 2^{n+2} (n+1)!^{4+3m}pq(p + q)$$
where $m$ is an integer greater than $\mu \ln(\Lambda /inj(M))$ and when $n>4$ we also require $m\geq 2^{n+1}$. The constant $\mu$ is as follows:
\begin{enumerate}
\item{When $M$ is Euclidean, $\mu = n+1$}
\item{When $M$ is Spherical, $\mu=2n+1$}
\item{When $M$ is Hyperbolic, $\mu = n\cosh^{n-1}(\Lambda)+1$}
\end{enumerate}
\end{theorem}

From Lemma \ref{inj} and Theorem \ref{HK} it follows that $inj(M) > \pi vol(M)/\delta vol(\S^n)$ which gives the following corollary in terms of the volume and diameter of the manifold.

\begin{corollary}\label{maincor}
With notations as in Theorem \ref{mainthm}, we can take $m$ to be an integer greater than $\mu \ln(\Lambda \delta vol(\S^n)/(\pi vol(M)))$ and when $n>4$ we also require $m \geq 2^{n+1}$. The constant $\delta$ is as follows:
\begin{enumerate}
\item{When $M$ is Euclidean, $\delta = diam(M)$}
\item{When $M$ is Spherical, $\delta=\sin^{n-1}(diam(M))$}
\item{When $M$ is Hyperbolic, $\delta = \sinh^{n-1}(diam(M))$}
\end{enumerate}
\end{corollary}

To express $m$ entirely in terms of the triangulation, we therefore need an upper diameter bound and a lower volume bound as a function of $n$, $p$, $q$ and bounds on lengths of edges. We therefore replace $diam(M)$ by $p\Lambda$, because we can always choose a piecewise geodesic path between two points of $M$ that intersects each simplex at most once and by Lemma \ref{diam}, the diameter of a simplex is bounded by the maximum length of its edges. And we can replace $vol(M)$ by $p$ times the volume of a regular $n$-simplex of length $\lambda$, where $\lambda$ is a lower bound on the length of edges of the triangulation. This gives the following corollary with a bound entirely in terms of the combinatorics of the triangulations and bounds on lengths of edges:

\begin{corollary}\label{maincor2}
Let $M$ be closed spherical, Euclidean or hyperbolic $n$-manifold with geometric triangulations $K_1$ and $K_2$. Let $K_1$ and $K_2$ have $p$ and $q$ many $n$-simplexes respectively with lengths of edges bounded above by $\Lambda$ and below by $\lambda$. When $M$ is spherical, we require $\Lambda \leq \pi/2$.  Let $\Delta^n_\lambda$ denote the regular $n$-simplex with edges of length $\lambda$.

When $n\leq 4$, then $K_1$ and $K_2$ are related by $f(n, p, q, \Lambda, \lambda)$ many Pachner moves which do not remove common vertices. In general, their $2^{n+1}$-th barycentric subdivisions, $\beta^{2^{n+1}} K_1$ and $\beta^{2^{n+1}} K_2$ are related by $f(n, p, q, \Lambda, \lambda)$ many Pachner moves which do not remove common vertices. The bounding function $f$ is the following: 
$$f(n, p, q, \Lambda, \lambda)= 2^{n+2} (n+1)!^{4+3m}pq(p + q)$$
where $m$ is an integer greater than $\mu \ln(\Lambda \delta vol(\S^n)/(\pi p\, vol(\Delta^n_\lambda))$ and when $n>4$ we also require $m\geq 2^{n+1}$. The constants $\mu$ and $\delta$ are as follows:
\begin{enumerate}
\item{When $M$ is Euclidean, $\mu = n+1$, $\delta = p\Lambda$}
\item{When $M$ is Spherical, $\mu=2n+1$, $\delta=\sin^{n-1}(p\Lambda)$}
\item{When $M$ is Hyperbolic, $\mu=n\cosh^{n-1}(\Lambda)+1$, $\delta = \sinh^{n-1}(p\Lambda)$}
\end{enumerate}
\end{corollary}

\begin{remark}\label{volhyp}
Hyperbolic manifolds have lower volume bounds that depend only on the dimension of the manifold. This allows us to obtain a lower bound for $m$ independent of $\lambda$. For $M$ a closed orientable hyperbolic 3-manifold, volume is bounded below by $w=0.9427$ \cite{GabMeyMil}, so using Corollary \ref{maincor} we can take 
$$m>(3\cosh^2(\Lambda) +1)\ln(2 \pi \Lambda \sinh^2(p\Lambda)/w)$$
For even dimensional closed hyperbolic manifolds, Hopf's generalised Gauss Bonnet formula gives us $vol(M)= (-1)^{n/2} vol(\S^n) \chi(M)/2$ where $\chi(M)$ is the Euler characteristic of $M$, so using Corollary \ref{maincor} again we can take 
$$m > \max(2^{n+1},(n\cosh^{n-1}(\Lambda) +1) \ln(2\Lambda \sinh^{n-1}(p\Lambda)/\pi))$$
In general for closed hyperbolic $n$-manifolds, volume is universally bounded below by $vol(\S^{n-1})/(n(n+3)^n \pi^{n(n-1)})$ \cite{Kel}. As $vol(\S^n)/(\pi vol(\S^{n-1})) < 1$ for all $n>1$, we can take 
$$m > \max(2^{n+1}, (n\cosh^{n-1}(\Lambda) +1) \ln(\Lambda \sinh^{n-1}(p\Lambda) n (n+3)^n \pi^{n(n-1)}))$$
\end{remark}



We must point out that as Pachner moves are combinatorial in nature, the intermediate triangulations we obtain may not be geometric. But as they are just local combinatorial operations, such a bound gives a naive algorithm to check if given hyperbolic or low dimensional spherical manifolds are isometric.

\begin{corollary}\label{isomcor}
Let $(M, K_M)$ and $(N, K_N)$ be geometrically triangulated closed hyperbolic manifolds of dimension at least $3$ or closed spherical manifolds of dimension at most $6$ and edge length at most $\pi/2$. Let $p$ and $q$ be the number of $n$-simplexes in $K_M$ and $K_N$ respectively. Let $\lambda$ and $\Lambda$ be a lower and upper bound on the lengths of edges of $K_M$ and $K_N$. Then $M$ is isometric to $N$ if and only if the $2^{n+1}$-th barycentric subdivisions of $K_M$ and $K_N$ are related by $f(n, p, q,  \Lambda, \lambda)$ Pachner moves followed by a simplicial isomorphism, with $f$ as defined in Corollary \ref{maincor2}.
\end{corollary}

We show that geometric triangulations can be related by geometric Pachner moves in \cite{KalPha2} using simplicial cobordisms. We follow a different approach in this article using shellings and relating via combinatorial Pachner moves instead as it leads to a tighter bound.

\subsection{Outline of Proof}
Given geometric triangulations $K_1$ and $K_2$ of $M$, we first take repeated barycentric subdivisions until each simplex lies in a strongly convex ball. This is where we crucially use the upper length bound on the edges to handle tall thin 'needle-shaped' tetrahedra. Next we consider the geometric polyhedral complex $K_1 \cap K_2$ obtained by intersecting the simplexes of $K_1$ and $K_2$, which we further subdivide to a common geometric subdivision $K'$. As simplexes of $K_1$ and $K_2$ are strongly convex they intersect at most once, which gives a bound on the number of simplexes in $K'$.

While every simplex of $K'$ lies in some simplex of $K_i$, to see that it is in fact a simplicial subdivision (and hence $K_i$ are PL-equivalent) we would need an embedding of $K_i$ in some $\R^n$ which is linear on both simplexes of $K_i$ and of $K'$. For example, there exists a simplicial topological triangulation of a 3-simplex $\Delta$ which contains a trefoil with just 3 edges in its 1-skeleton \cite{Lic2}. As the stick number of a trefoil is 6, $K'$ is a topological triangulation but not a simplicial subdivision of $\Delta$. For constant curvature manifolds however, there do exist local embeddings in $\R^n$ which are linear on both $K_i$ and $K'$. This allows us to treat geometric subdivisions of geometric simplexes in the manifold as simplicial subdivisions of linear simplexes in $\R^n$.

Shelling of a triangulated polytope, introduced in the seminal 1971 paper of Bruggesser and Mani\cite{BruMan}, is a way to inductively remove simplexes $\sigma_i$ from the triangulation such that at each stage $\del \sigma_i$ intersects what remains in a pure (n-1)-dimensional complex. It is easy to see that 2 dimensional polytopes are shellable. Higher dimensional PL polytopes are not in general shellable. The earliest example of nonshellable topological subdivisions of 3-polytopes were given by Newman\cite{New} way back in 1926. Later Rudin\cite{Rud} showed that even linear subdivisions of a 3-simplex may not be shellable. For spheres, Lickorish\cite{Lic2} has given several examples of unshellable triangulations. These examples illustrate that even in the simplest of cases, the property of shellability may not hold. Recently though Adiprasito and Benedetti\cite{AdiBen} have shown that linear subdivisions of convex polytopes are shellable up to subdivision.
\begin{theorem}[Theorem A of \cite{AdiBen}]\label{AdiBenThm}
If $C$ is any subdivision of a convex polytope, the second derived subdivision of $C$ is shellable. If $\dim C=3$, already the first derived subdivision of $C$ is shellable.
\end{theorem}

We first take repeated barycentric subdivisions to make the link of every simplex of $K$ shellable. Given a geometric subdivision $\alpha K$ of $K$, we next define partial barycentric subdivisions $\beta^\alpha_r K$ by putting the given subdivision $\alpha A$ on simplexes $A$ of dimension at most $r$ and the barycentric subdivision $\beta A$ on the rest.  By Theorem \ref{AdiBenThm}, $\alpha A$ is shellable up to subdivisions and as link of $A$ in $\beta^\alpha_r K$ is also shellable so we can extend shellability to 'star neighbourhoods' of $\alpha A$ in $\beta^\alpha_r K$. When a polytope is shellable it is easy to see that it is starrable, i.e., there exists a sequence of Pachner moves which takes the subdivision of a star neighbourhood to the cone over its boundary. Using this, we get a sequence of Pachner moves which takes a star neighbourhood of $\alpha A$ to a cone on its boundary and varying $A$ over all $r$ simplexes of $K$, a sequence of moves from $\beta^\alpha_r K$ to $\beta^\alpha_{r-1} K$. This gives a sequence of moves from $\beta^\alpha_n K = \alpha K$ to $\beta^\alpha_0 K = \beta K$ (as in Figure \ref{Moves}). Taking $\alpha K$ as the common geometric subdvision $K'$ of $K_1$ and $K_2$, we get a sequence of moves from $\beta K_1$ to $\beta K_2$ of controlled length as required.

\subsection{Organisation of paper.}
In Section \ref{linksec} we fix notation, give the necessary definitions of PL topology and prove that after taking sufficiently many barycentric subdivisions the link of a simplex of a geometric triangulation is shellable. In Section \ref{subdivsec} we intersect the two given geometric triangulations to obtain a common geometric subdivision. By repeatedly taking barycentric subdivisions we ensure that all simplexes are strongly convex. This allows us to control the number of simplexes in this common subdivision. We combine the shellability of links obtained in Section \ref{linksec} with the shellability of subdivisions of simplexes obtained via Theorem \ref{AdiBenThm} to obtain a sequence of Pachner moves of controlled length that changes the star neighbourhood of simplexes to cones over their boundaries. This leads to a proof of Theorem \ref{mainthm}. Section \ref{scalingsec} involves extensive application of Euclicean, spherical and hyperbolic trigonometry to calculate the factor by which barycentric subdivisions scale simplexes.

\section{Links of simplexes}\label{linksec}
In this section we fix notation and prove results about links of partial barycentric subdivisions of simplicial complexes. The main result of this section is that after taking sufficiently many barycentric subdivisions, the link of every simplex of a geometric triangulation is shellable. Books by Rourke and Sanderson\cite{RouSan} and Ziegler\cite{Zie2} are good sources of introduction to the theory of piecewise linear topology.

A \emph{simplicial complex} contains a finite set $K^0$ (the vertices) and a family $K$ of subsets of $K^0$ (called simplexes) such that if $B \subset A \in K$ then $B \in K$. If the maximal simplex is of size $n+1$, we call $n$ the dimension of $K$. A \emph{simplicial isomorphism} between two complexes is a bijection between their vertices which induces a bijection between their simplexes. A \emph{realisation} of a simplicial complex $K$ is a subspace $|K|$ of some $\R^N$, where the vertices of each simplex are affinely independent and represented by the linear simplex which is their convex hull. We call $K$ a \emph{triangulation} of a manifold $M$ when there exists a homeomorphism from a realisation $|K|$ of $K$ to $M$. The simplexes of this triangulation are the images of simplexes of $|K|$ via this homeomorphism. Every simplicial complex has a realisation in $\R^{N}$ where $N$ is the size of $K_0$, by representing $K_0$ as a basis of $\R^N$. Any two realisations of a simplicial complex are simplicially isomorphic. For $A$ a simplex of $K$, we denote by $\del A$ the boundary complex of proper faces of $A$. When the context is clear, we shall use the same symbol $A$ to denote the simplex and the simplicial complex $A \cup \del A$.

\begin{definition}
Let $C$ be a subset $C$ of a Riemannian manifold $M$. We call $C$ \emph{convex}, if any two points in $C$ are connected by a unique geodesic in $C$. We call $C$ \emph{star convex} if there exists a point $p$ in interior of $C$ such that there is a unique  geodesic in $C$ from $p$ to any point in $\del C$. We call it \emph{strongly convex}, if any two points in $C$ are connected by a unique minimising geodesic in $M$ which also happens to lie entirely in $C$.
\end{definition}

\begin{definition}
For $A$ and $B$ simplexes of a simplicial complex $K$, we denote their \emph{join} $A\star B$ as the simplex $A \cup B$. The join operation makes sense when $A$ and $B$ are totally geodesic disks in a convex set in a constant curvature manifold, as the join of totally geodesic disks give a totally geodesic disk. The \emph{link} of a simplex $A$ in a simplicial complex $K$ is defined by $lk (A, K) = \{ B \in K : A \star B \in K \}$. The (closed) \emph{star} of $A$ in $K$ is defined by $st(A, K)=A \star lk (A, K)$.
\end{definition}

\begin{figure}
\centering
\includegraphics[width=0.6\textwidth]{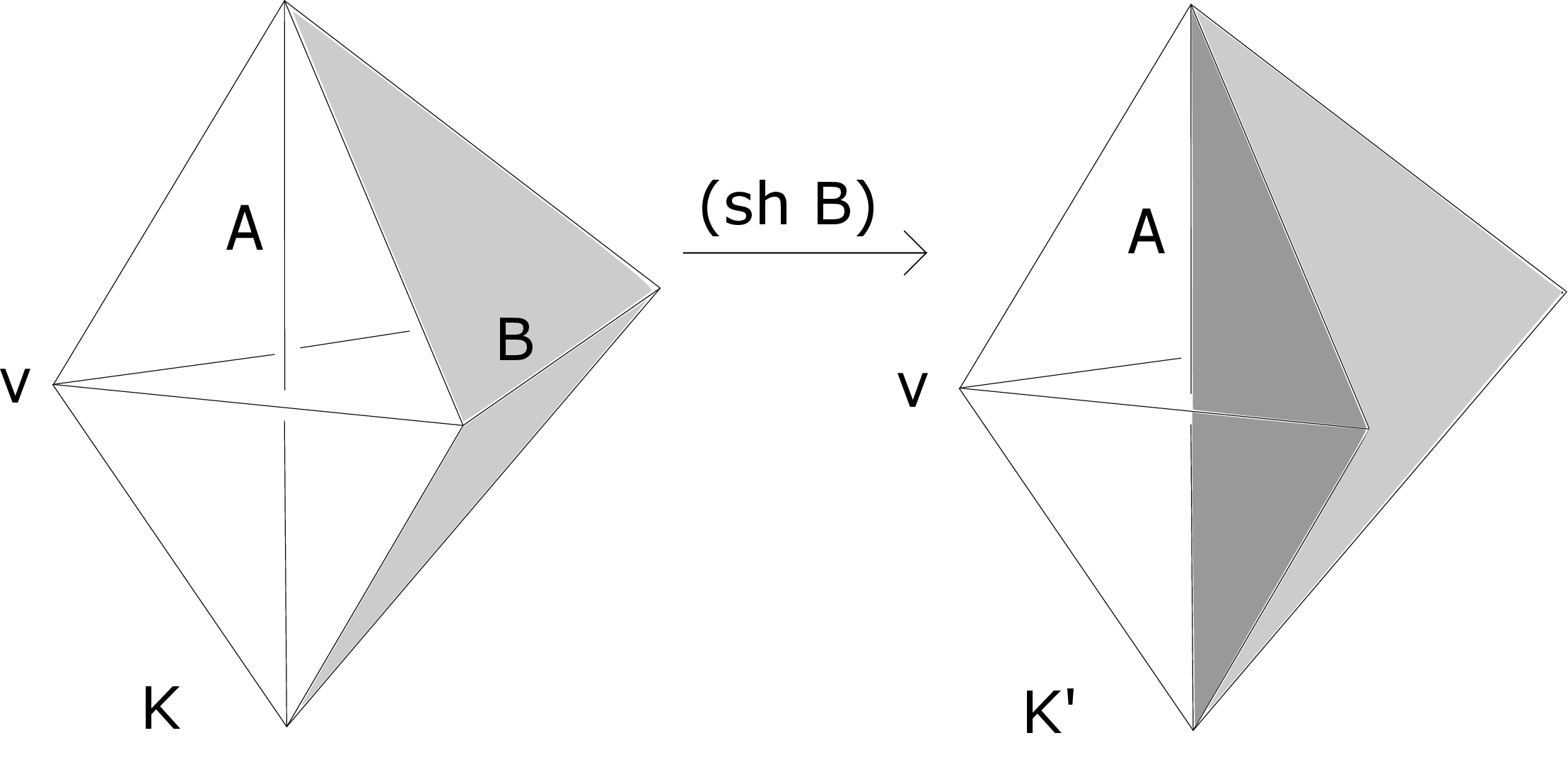}
\caption{$B \star \del A \subset \del K$ which shells along $B$ to a complex $K'$ by removing $A\star B$ from $K$ and taking closure.}\label{shelling}
\end{figure}


\begin{definition}\cite{Lic}
Suppose that $A$ and $B$ are simplexes of a simplicial triangulation $K$ of an $n$-manifold $M$ with boundary $\del M$, that $A \star B$ is an $n$-simplex of $K$, that $A \cap \del M = \del A$ and that $B\star \del A \subset \del K$. Then the simplicial complex $K'$ obtained from $K$ by elementary shelling along $B$ is the closure of $K \setminus (A \star B)$. Closure here means adding the simplexes of $A \star \del B$. The relation between $K$ and $K'$ will be denoted by $K \xrightarrow{(sh B)} K'$. See Figure \ref{shelling} for an example. An $n$-ball is said to be \emph{shellable} if it can be reduced to an $n$-simplex by a sequence of elementary shellings. An $n$-sphere is shellable if removing some $n$-simplex from it gives a shellable $n$-ball. 
\end{definition}

\begin{definition}
Suppose that $A$ is an $r$-simplex in a simplicial complex $K$ of dimension $n$ and that $lk (A, K) = \del B$ for some $n-r$ simplex $B \notin K$. Then the \emph{Pachner move} $\kappa(A, B)$ consists of changing $K$ by removing $A \star \del B$ and inserting $\del A \star B$. See Figures \ref{Pac2} and \ref{Pac3} for the possible $2$ and $3$-dimensional Pachner moves.
\end{definition}

\begin{figure}
\centering
\includegraphics[width=1\textwidth]{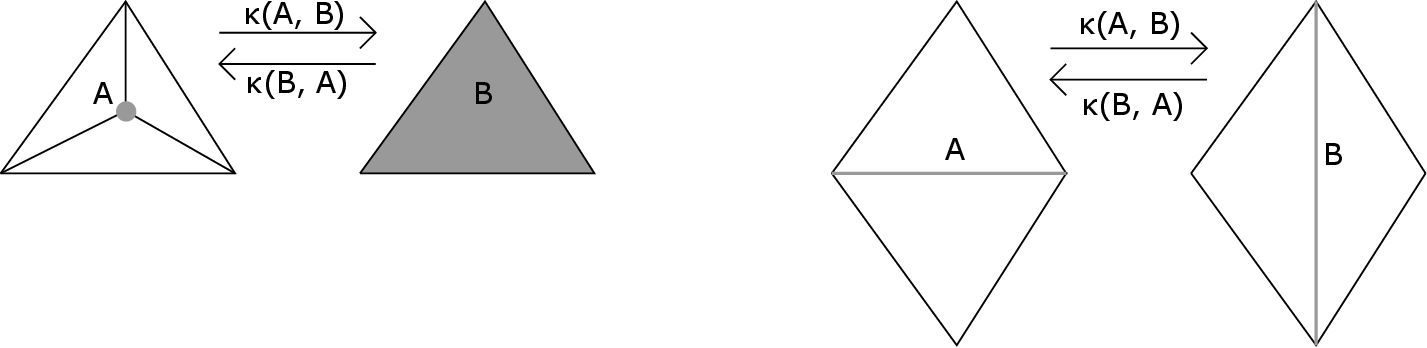}
\caption{2-dimensional Pachner moves}\label{Pac2}
\end{figure}

\begin{figure}
\centering
\includegraphics[width=1\textwidth]{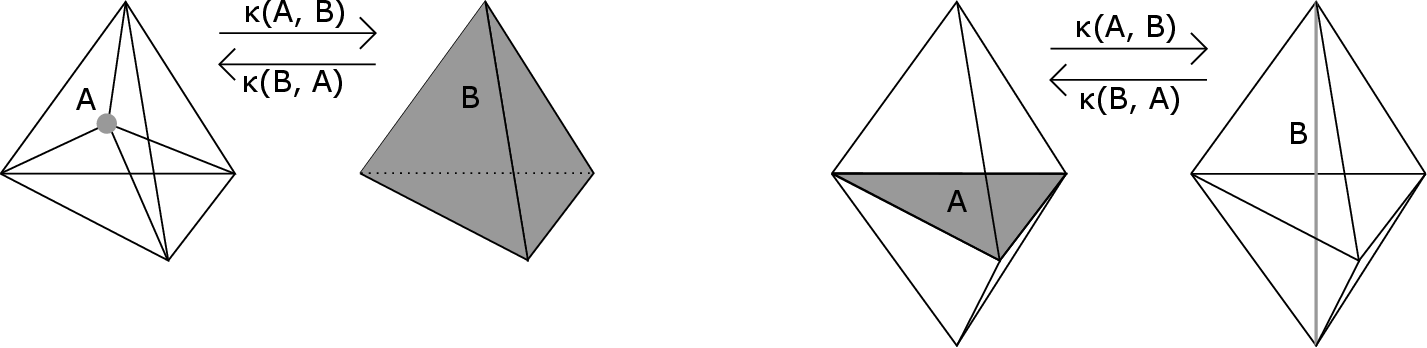}
\caption{3-dimensional Pachner moves}\label{Pac3}
\end{figure}

\begin{definition}
We say that an $n$-ball $K$ is \emph{starrable} if there exist a finite sequence of Pachner moves from $K$ to $p \star \del K$ for some point $p\in int(K)$.
\end{definition}

We reproduce the proof of the statement that shellable balls are starrable from \cite{Lic} for completeness and to record the number of Pachner moves required in the starring process.

\begin{lemma}[Lemma 5.7 of \cite{Lic}] \label{shellable}
Let $K$ be a shellable triangulation of an $n$-ball with $r$ many $n$-simplexes, then  $K$ is starrable by a sequence of $r$ Pachner moves.
\end{lemma}
\begin{proof}
We prove this by induction on the number $r$ of $n$-simplexes of $K$. If $r=1$, then $K$ is a $n$-simplex $B$ and let $A=p$ be a point in the interior of $B$. A single Pachner move $\kappa(B, A)$ changes $K$ to $p\star \del K$, as in the diagrams on the left in Figures \ref{Pac2} and \ref{Pac3}.

Suppose that the first elementary shelling of $K$ is $K \xrightarrow{(sh B)} K_1$, where $A \star B$ is a $n$-simplex of $K$, $A \cap \del K = \del A$ and $B \star \del A \in \del K$ (see Figure \ref{shelling}).  By the induction on $r$, $K_1$ is simplicially isomorphic to  $p\star \del K_1$ after at most $r-1$ Pachner moves. Observe that $p\star \del K_1 \cup A\star B$ is changed to $p\star \del K$ by the single Pachner move $\kappa (A, p\star B)$. 
\end{proof}

\begin{figure}
\centering
\includegraphics[width=0.9\textwidth]{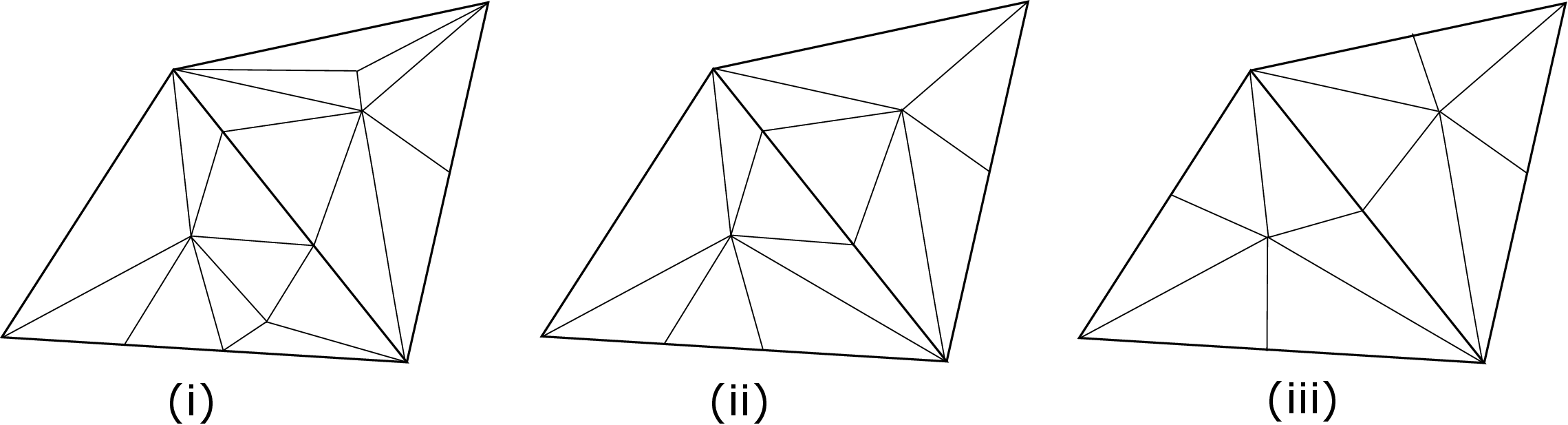}
\caption{In this example $K$ is the simplicial complex consisting of $2$ triangles in solid lines, (i) $\beta^\alpha_2 K =\alpha K$ is a subdivision of $K$, (ii) $\beta^\alpha_1 K$ and (iii) $\beta K$ is the barycentric subdivision of $K$.}\label{Moves}
\end{figure}

\begin{definition}\label{barydefn}
Let $\alpha K$ be a geometric subdivision of $K$. Let $\beta^\alpha_r K$ be the geometric subdivision of $K$ such that, if $A$ is a simplex in $K$ and $\dim (A) \leq r$, then $\beta^\alpha_r A = \alpha A$ and if $\dim (A) > r$ then $\beta^\alpha_r A = a \star \beta^\alpha_r \del \alpha A$ for some point $a \in int(A)$, i.e. it is subdivided as the geometric cone on the already defined subdivision of its boundary. In other words, fix a point $a \in int(A)$ and for each  simplex $B \in \beta_r \del \alpha A$, introduce the geometric simplex $a\star B$ by taking the union of geodesics in $A$ which start at $a$ and end at some point in $B$. When $M$ is Euclidean or hyperbolic then $A$ is strongly convex and when $M$ is spherical we can choose the point $a\in int(A)$ such that the distance in $A$ between $a$ and $\del A$ is less than $\pi$ so in either case there is a unique geodesic in $A$ between $a$ and every point in $B$. The geometric join of a point with a totally geodesic disk is again a totally geodesic disk, so $\beta^\alpha_r K$ is also a geometric triangulation of $M$. Observe that $\beta^\alpha_n K$ is $\alpha K$ while $\beta^\alpha_0 K = \beta K$ is the \emph{geometric barycentric subdivision} of $K$. See Figure \ref{Moves} for an example. When $\alpha K = K$, we denote $\beta^\alpha_r K$ by $\beta_r K$ and call it a \emph{partial barycentric subdivision}.
\end{definition}

The following lemma relates the links of simplexes in a partial barycentric subdivision with the barycentric subdivision of the links in the original simplicial complex, as can be seen in Figure \ref{barylinkfig}.
\begin{lemma}\label{barylink}
Let $A$ be an $r$-simplex in a simplicial complex $K$. Then $lk(A, \beta_r K)$ is simplicially isomorphic to $\beta lk(A, K)$.
\end{lemma}
\begin{proof}
Observe that as $A$ is $r$-dimensional, $\beta_r A = A$ and we can take $A$ to be a simplex of both $\beta_r K$ and $K$.

Let $B$ be a simplex in $lk(A, K)$. The barycentric subdivision $\beta B$ of $B$ is given by $b\star\beta \del B$. So the vertices of $\beta lk(A, K)$ are exactly such points $b$, one for each simplex $B$ in $lk(A, K)$. As $A\star B$ has dimension greater than $r$, so $\beta_r (A\star B) = b'\star \beta_r(\del (A\star B))$. And as $A$ is unchanged by $\beta_r$, so $A \in \beta_r(\del (A \star B))$ and consequently $b' \star A \in \beta_r(A \star B) \subset \beta_r K$. So given $B \in lk(A, K)$, we obtain a vertex $b'$ of $lk(A, \beta_r K)$. Conversely, given a vertex $b'$ of $lk(A, \beta_r K)$, $b'\star A$ is a simplex in $\beta_r K$ of dimension more than $r$. So there exists some $B \in lk(A, K)$ such that $\beta_r(A\star B) = b'\star \beta_r (\del (A\star B))$.

Define $\phi$ as this bijection from the vertex set of $\beta lk(A, K)$ to the vertex set of $lk(A, \beta_r K)$ which sends the vertex $b$ corresponding to $B\in lk(A, K)$ to the vertex $b'$ of $\beta_r(A\star B)$. We shall extend $\phi$ inductively to a simplicial isomorphism from $\beta lk(A, K)$ to $lk(A, \beta_r K)$. See Figure \ref{barylinkfig} for the case when $K=A\star B$ and $r=1$.

As $\phi$ is a bijection on the vertices it is a simplicial isomorphism on the $0$-skeleton of $\beta lk(A, K)$. Let $B \in lk(A, K)$ be $m$ dimensional and assume that $\phi$ is a simplicial isomorphism on the $m-1$ skeleton of $\beta lk(A, K)$. As $\beta_r(A\star B) = b' \star \beta_r \del(A \star B) = (b'\star\beta_r (\del A \star B)) \cup (b'\star \beta_r(A \star \del B))$ so each simplex of $\beta_r(A\star B)$ lies entirely in $(b'\star \beta_r (\del A\star B))$ or $(b'\star\beta_r(A\star \del B))$ (or both). So if $A\star C \in \beta_r(A\star B)$ then as $A$ belongs to $b'\star\beta_r(A\star \del B)$ and $A \notin b' \star \beta_r(\del A \star B)$ therefore $C$ belongs to it as well, and we get $lk(A, \beta_r(A\star B))=lk(A, b'\star \beta_r(A \star \del B))$. As $A \in \beta_r(A \star \del B)$ so $lk(A, b'\star \beta_r(A \star \del B)) = b'\star lk(A, \beta_r(A\star \del B))$. By assumption, $\phi$ restricted to $\beta(\del B)$ is simplicially isomorphic to $lk(A, \beta_r(A \star \del B))$. So  $\beta B = b\star\beta(\del B)$ is simplicially isomorphic via $\phi$ to $b'\star lk(A, \beta_r(A\star\del B))=lk(A, \beta_r(A\star B))$. Varying $B$ over all $m$-simplexes, shows that $\phi$ is a simplicial isomorphism on the $m$-skeleton of $\beta lk(A, K)$. So by induction taking $m=n$, we get a simplicial isomorphism from $\beta lk(A, K)$ to $lk(A, \beta_r K)$.
\end{proof}

\begin{figure}
\centering
\includegraphics[width=0.3\textwidth]{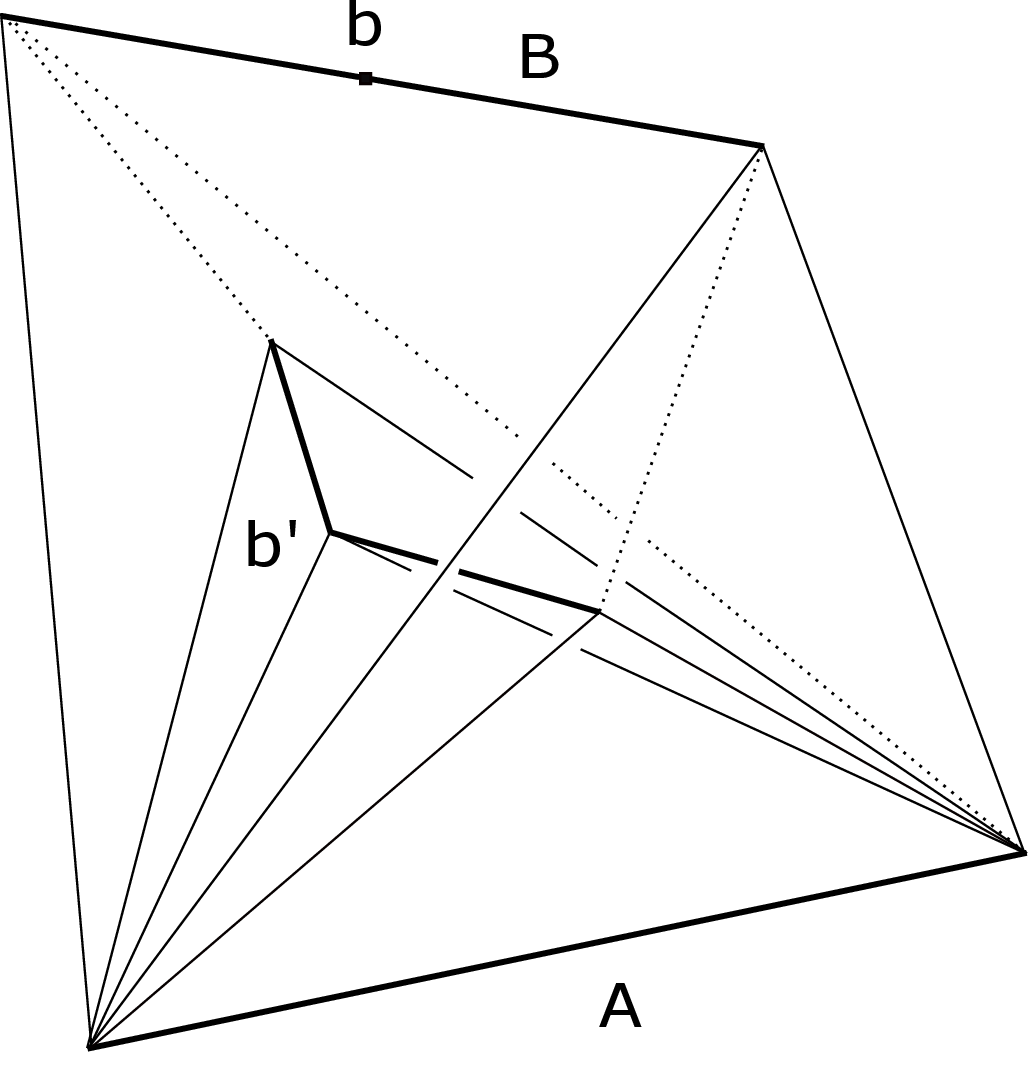}
\caption{When $K= A\star B$, $\beta B$ is isomorphic to $lk(A, \beta_1 K)$}\label{barylinkfig}
\end{figure}

The following result proved in \cite{Zee} for simplicial complexes  also works for spherical triangulations when diameter of each simplex is less than $\pi$ so that any two simplexes intersect at most once:
\begin{lemma}\label{baryrefinement} [Lemma 4, Ch 1 \cite{Zee}]
Let $K$ and $L$ be geometric triangulations of a sphere, with diameter of each simplex less than $\pi$. Then for $s$ the total number of simplexes of $K$, the $s$-th derived subdivision of $L$ is a subdivision of $K$.
\end{lemma}

\begin{lemma}\label{simplexlink}
Let $K$ be a geometric triangulation of a Riemannian $n$-manifold. For all vertices $v$ of $K$, $\beta^{m-2} lk(v, K)$ is simplicially isomorphic to a subdivision of the boundary of an $n$-simplex, where $m=2^{n+1}$.
\end{lemma}
\begin{proof}
As $K$ is a geometric triangulation, $|int(st(v, K))|$ is a neighbourhood of $v$. Let $L \subset int(st(v, K))$ be a geometrically triangulated sphere centered at $v$ which is isomorphic to $lk(v, K)$ under radial projection. Let $\delta$ be a spherical triangulation of $|L|$ as the boundary of a spherical $n$-simplex, so in particular $\delta$ has $m-2$ simplexes. By Lemma \ref{baryrefinement}, $\beta^{m-2} L$ is a subdivision of $\delta$ which is isomorphic to the boundary of a simplex.
\end{proof}

\begin{theorem}\label{shellinglink}
Links of all simplexes in $\beta^{m} K$ are shellable for $m=2^{n+1}$.
\end{theorem}
\begin{proof}
We first prove that all vertex links in $\beta^m K$ are shellable. We will repeatedly use the fact that $lk(v, \beta K)\simeq \beta lk(v, K)$.  

For $v$ a vertex of $K$, by Lemma \ref{simplexlink} $st(v, \beta^{m-2} K) = v\star lk(v, \beta^{m-2} K) \simeq v \star \beta^{m-2} lk(v, K)$ is isomorphic to the subdivision of a simplex. By Theorem A of \cite{AdiBen} then, $\beta^2 st(v, \beta^{m-2} K)$ is shellable. As links of vertices of shellable complexes are shellable, so $lk(v, \beta^2 st(v, \beta^{m-2} K)) = \beta^2 lk(v, st(v, \beta^{m-2}K))=\beta^2 lk(v, \beta^{m-2} K)= lk(v, \beta^m K)$ is shellable for any vertex $v\in K$.

As $\beta^m (v \star lk(v, K))$ is the stellar subdivision of $v\star \beta^m lk(v, K)$ and as stellar subdivisions preserve shellability (see proof of Proposition 1 of \cite{BruMan}) so $\beta^m (st(v, K))$ is shellable. For any vertex $w \in int(\beta^m st(v, K))$, $lk(w, \beta^m K)=lk(w, \beta^m st(v, K))$ is shellable as it is the link of a vertex in a shellable complex. As  $|K| = |\beta^m K|$ is covered by $|int(st(v, K))|=|int(\beta^m st(v, K))|$ so the link of any vertex $w$ of $\beta^m K$ is shellable.

Having shown that all vertex links in $\beta^m K$ are shellable, we will now show by induction on $r$ that links of all $r$-simplexes in $\beta^m K$ are shellable. Assume that links of all simplexes of dimension less than $r>0$ are shellable. Let $A=B\star b$ be an $r$-simplex in $\beta^m K$ for a vertex $b$.  We claim that $lk(A, \beta^m K)=lk(b, lk(B, \beta^m K))$. Let $C \in lk(A, \beta^m K)$ then, by definition of links, $C \star (B\star b)$ is a simplex of $\beta^m K$. As $C \star b \in lk(B, \beta^m K)$, so $C \in lk(b, lk(B, \beta^m K))$. Conversely if $C \in lk(b, lk(B, \beta^m K))$ then $C\star b$ is a simplex in $lk(B, \beta^m K)$ and therefore $(C\star b) \star B =C \star A \in \beta^m K$. Hence, $C \in lk(A, \beta^m K)$. 

As $B$ is a simplex of dimension less than $r$, so by induction $lk(B, \beta^m K)$ is shellable and as links of vertices of shellable complexes are shellable, so $lk(A, \beta^m K) = lk(b, lk(B, \beta^m K))$ is also shellable.

\end{proof}

We end this section with a similar statement about links of simplexes of partial barycentric subdivisions.

\begin{lemma}\label{links}
Let $K$ be a simplicial complex such that the link of each vertex is shellable. Let $A$ be an $r$-simplex in $K$, then $lk(A, \beta_r K)$ is shellable.
\end{lemma}
\begin{proof}
For $A$ any simplex of $K$, $lk(A, \beta_r K)\simeq \beta lk(A, K)$ by Lemma \ref{barylink}. By arguments as in Theorem \ref{shellinglink}, as vertex links of $K$ are shellable, so $lk(A, K)$ is shellable. And as the barycentric subdivision of a shellable complex is shellable (see Theorem 5.1 of \cite{Bjo}) so $lk(A, \beta_r K)$ is shellable.
\end{proof}

\section{Common geometric subdivision}\label{subdivsec}
Given two abstract simplicial complexes, there is no canonical notion of a common subdivision. In this section we use the geometry of the manifold to get a common geometric subdivision of two geometric triangulations. This allows us to relate them via a bounded length sequence of Pachner moves through the common subdivision. We must caution here that even though the terminal triangulations of this sequence are geometric in nature, the intermediate triangulations we obtain are merely topological triangulations.

 A hyperbolic, spherical or Euclidean $k$-simplex in $\H^n$, $\S^n$ or $\E^n$ is the convex hull of a generic set of $k+1$ points. In the spherical case, we further assume that the diameter of the simplex is at most $\pi/2$.

\begin{definition}
A \emph{geometric simplicial triangulation} $K$ of a hyperbolic, spherical or Euclidean manifold $M$ is a simplicial triangulation of $M$ where each simplex is isometric to a hyperbolic, spherical or Euclidean simplex respectively. We say a geometric simplicial triangulation $K'$ of $M$ is a \emph{geometric subdivision} of $K$ if each simplex of $K'$ is isometrically embedded in some simplex of $K$.
\end{definition}

As mentioned in the introduction, when $M$ is a closed Euclidean or hyperbolic manifold or a spherical manifold of diameter less than $\pi$, then $M$ has a geometric triangulation. We henceforth fix the notation $(M, K)$ to refer to the geometric simplicial triangulation $K$ of a closed hyperbolic, spherical or Euclidean manifold $M$ of dimension $n$. The following simple observation spelt out in \cite{KalPha2} allows us to treat the geometric triangulation of a convex polytope in $M$ as the Euclidean triangulation of a convex polytope in $\E^n$. 

\begin{lemma}\cite{KalPha2} \label{convex}
Let $K$ be a geometric simplicial triangulation of a constant curvature $n$-manifold  $M$. When $M$ is spherical we require each simplex to have diameter less than $\pi$. Then each simplex of $K$ is homeomorphic to a linear simplex in $\E^n$ by a map which takes geodesics to straight lines.
\end{lemma}

An outline of the proof is that every geometric simplex lifts to either $\E^n$, the Klein model of $\H^n$ or a hemisphere in $\S^n$ followed by the radial/gnomonic projection. In either case we get a map from the geometric simplex in $M$ to a linear simplex in $\E^n$ which takes geodesics to straight lines. So in particular, it takes a geometric subdivision of the simplex to a simplicial subdivision of the corresponding linear simplex.

\begin{lemma}\label{barycentric}
When $K$ has $p_i$ many $i$-simplexes, $\beta K$ has $(i+1)!p_i$ many $i$-simplexes in the $i$-skeleton of $K$.
\end{lemma}
\begin{proof}
To obtain the barycentric subdivision $\beta K$ of $K$ we replace each simplex of $K$ with the cone on its boundary, starting with vertices and inductively going up to simplexes of dimension $n$. 

For an $i$-simplex $A$, let $a_i$ be the number of $i$ simplexes in $\beta A$. As there are $i+1$ many codimension one faces of $A$ so $a_i = (i+1) a_{i-1}$ and $a_0 = 1$. This gives $a_i = (i+1)!$. So if there are $p_i$ many $i$-simplexes in $K$, there are $(i+1)! p_i$ many $i$-simplexes of $\beta K$ in the $i$-skeleton of $K$.
\end{proof}

We say that two simplicial complexes are \emph{stellar equivalent} if they are related by a sequence of stellar moves. We call them \emph{bistellar equivalent} if they are related by a sequence of Pachner moves. The following is an effective version of Lemma 4.4 of \cite{Lic} using the stronger notion of shellability instead of starrability, to get bistellar equivalence in place of stellar equivalence.

\begin{lemma}\label{induction}
Let $K$ be a geometric triangulation where the link of every positive dimensional simplex is shellable. Let $\alpha K$ be a geometric subdivision of $K$ such that for each simplex $A \in K$, $\alpha A$ is shellable. Let $p_i$ be the number of $i$-simplexes of $K$, with $p_{-1}=1$. Let $s_i$ be the number of $i$-simplexes of $\alpha K$ in the $i$-skeleton of $K$. Then $\alpha K$ is related to $\beta K$ by $\sum_{i=1}^{n} (n-i)! p_{n-i-1}s_{i}$ Pachner moves. Furthermore, none of these Pachner moves remove any vertex of $K$.
\end{lemma}

\begin{proof}
Our aim is to bound the number of Pachner moves needed to go from $\beta^\alpha_{r}K$ to $\beta^\alpha_{r-1}  K$ for $1\leq r \leq n$. This would give us a bound on the number of moves relating $\beta^\alpha_n K = \alpha K$ and  $\beta^\alpha_{0} K = \beta K$. See Figure \ref{Moves}.

As links of simplexes in $K$ are given to be shellable, so for any $r$-simplex $A\in K$, by Lemma \ref{links}, $lk(A, \beta_r K) \simeq \beta lk(A, K)$ is shellable. As $\alpha A$ is given to be shellable so $S(A) = \alpha A \star lk(A, \beta_r K)$, the join of shellable complexes, is shellable as well. $S(A)$ should morally be thought of as the star neighbourhood of $\alpha A$ in $\beta_r K$. 

Let $m_A$ be the number of $r$-simplexes of $\alpha A$ in $A$. The number of $(n-r-1)$ simplexes in $lk(A, K)$ is at most $p_{n-r-1}$, so by Lemma \ref{barycentric}  the number of $(n-r-1)$ simplexes in $\beta lk(A, K)$ is at most $(n-r)! p_{n-r-1}$. By Lemma \ref{barylink}, $\beta lk(A, K) \simeq lk(A, \beta_r K)$, so $S(A)$ has at most $(n-r)!p_{n-r-1}m_A$ many $n$-simplexes. 

By Lemma \ref{shellable}, there is a sequence of as many Pachner moves which changes $S(A)$ to $a \star \del S(A) = a\star \del \alpha A \star lk(A, \beta_r K)$, for $a$ a point in the interior of $A$. Making this change for each $r$-simplex $A$ of $K$ replaces each $\alpha A$ with $a \star \del \alpha A=a \star \beta^\alpha_{r-1}\del \alpha A$ while higher dimensional simplexes of $K$ remain subdivided as cones on their boundary. This gives us $\beta^\alpha_{r-1} K$ from $\beta^\alpha_r K$ by at most $(n-r)! s_r p_{n-r-1}$ Pachner moves, where $s_r$ is the total number of $r$-simplexes of $\alpha K$ in the $r$-skeleton of $K$. So $\beta^\alpha_n K = \alpha K$ is related to $\beta^\alpha_0 K = \beta K$ by $\sum_{r=1}^{n} (n-r)! s_r p_{n-r-1}$ Pachner moves. When $r=n$, $lk(A, \beta_r K)$ is empty so we take $p_{-1}=1$. Note that as none of these Pachner moves remove any vertices of $A$ so they never remove any vertex of $K$.
\end{proof}

\begin{lemma}\label{barymoves}
Let $K$ be a simplicial complex where link of every positive dimensional simplex is shellable. The $m$-th barycentric subdivision $\beta^m K$ is related to $K$ by $(n+1)!^{2m+2} p_n^2$ Pachner moves.
\end{lemma}
\begin{proof}
Expanding each term of $(1+x)^n (1+x)^n = (1+x)^{2n}$ and comparing the coefficient of $x^n$ on both sides, gives us $\sum_{i=0}^n \binom{n}{i}\binom{n}{n-i} = \binom{2n}{n}$. Taking $\alpha K = K$ in Lemma \ref{induction}, for $A$ a simplex of $K$, $\alpha A =A$ is trivially shellable. Also $s_i=p_i$, so that $K$ is related to $\beta K$ by $\sum_{i=1}^n (n-i)! p_{n-i-1} p_i$ many Pachner moves. Bounding $p_i$ by $\binom{n+1}{i+1} p_n$ and $(2n+2)!$ by $4(n+1)!^3$ we get:

$$ \begin{array}{lllll}
\sum_{i=1}^n (n-i)! p_i p_{n-i-1} &< &\sum_{i=1}^n (n-i)! \binom{n+1}{i+1} \binom{n+1}{n-i} p_n^2 &< & (n-1)! p_n^2 \binom{2n+2}{n+1}\\
& < &  \frac{4}{n(n+1)}(n+1)!^2 p_n^2 & < & (n+1)!^2 p_n^2\\
\end{array}
$$
By Lemma \ref{barycentric}, the number of $n$-simplexes $p_n$ changes to $(n+1)!p_n$ on taking a barycentric subdivision. Therefore on taking $m$ subdivisions the bound on the number of moves relating $K$ and $\beta^m K$ becomes:
$$p_n^2 [(n+1)!^2 + (n+1)!^4 + ... + (n+1)!^{2m}] < (n+1)!^{2(m+1)}p_n^2$$
\end{proof}

We now use Theorem A of \cite{AdiBen} to bound the number of Pachner moves needed to relate a locally shellable geometric triangulation with its subdivision.

\begin{theorem}\label{mainlemma} 
Let $K$ be a geometric triangulation where the link of every positive dimensional simplex is shellable. Let $K'$ be a geometric subdivision of $K$. Let $p_i$ be the number of $i$-simplexes of $K$ for $i>0$, with $p_{-1} = 1$. Let $s_i$ be the number of $i$-simplexes of $K'$ that lie in the $i$-skeleton of $K$. Then $\beta^2 K'$ is related to $\beta K$ by $\sum_{i=1}^{n} (n-i)! (i+1)!(i+1)! p_{n-i-1}s_{i}$ many Pachner moves none of which remove any vertex of $K$. 
\end{theorem}
\begin{proof} 
By Lemma \ref{barycentric}, $\beta K'$ has less than $(i+1)!s_i$ many $i$-simplexes in the $i$-skeleton of $K'$ and applied a second time, $\beta^2 K'$ has less than $(i+1)! (i+1)!s_i$ many $i$-simplexes in the $i$-skeleton of $K'$. 

Let $\alpha K = K'$. For each simplex $A$ of $K$, by Lemma \ref{convex} there is a simplicial isomorphism from $\alpha A$ to a linear subdivision of a convex polytope in $\E^n$. By Theorem A of \cite{AdiBen}, its second barycentric subdivision $\beta^2 \alpha A$ is shellable and so replacing $s_i$ in Lemma \ref{induction} with $(i+1)!(i+1)!s_i$ we get the required bounds.

\end{proof}

In the rest of this section, we obtain a common subdivision with a controlled number of simplexes from a given pair of geometric triangulations.
\begin{definition}
Given a Riemannian manifold $M$, a \emph{geometric polytopal complex} $C$ of $M$ is a finite collection of geometric convex polytopes in $M$ whose union is all of $M$ and such that for every $P \in C$, $C$ contains all faces of $P$ and the intersection of two polytopes is a face of each of them.
\end{definition}

When each simplex of the geometric triangulations is strongly convex, any two simplexes intersect at most once. We can therefore bound the number of simplexes in the common geometric subdivision $K' = \beta(K_1 \cap K_2)$.

\begin{lemma}\label{commonsub}
When $K_1$ and $K_2$ are strongly convex geometric triangulations with $p_i$ and $q_i$ many $i$-simplexes respectively, then they have a common geometric subdivision $K'$ with $s_i$ many $i$-dimensional simplexes that lie in the $i$-skeleton of $K_1$ where $$s_i < (2^n-1) (n+1)!^2 p_i q_n$$
\end{lemma}
\begin{proof}
Let $A$ be a linear $k$-simplex and $B$ a linear $l$-simplex in $\R^N$. Suppose that $B$ intersects $A$ in a $k$-dimensional polytope $P$. So $l \geq k$ and the interiors of $A$ and $B$ intersect transversally inside a subspace $V(A+B)$ of $\R^N$ spanned by vectors in $A$ and $B$ (assume $0 \in A\cap B$). As their intersection $P$ is $k$ dimensional, so $V(P)=V(A) \cap V(B)$ is a $k$-dimensional space and by the Rank-Nullity theorem $V(A+B)$ is $l$-dimensional. Therefore any $(k-1)$-face of $P$ is obtained by intersecting an $(l-1)$ simplex of $B$ with the $k$-simplex of $A$ or by intersecting a $(k-1)$-simplex of $A$ with the $l$-simplex of $B$. There are therefore at most $(k+1) + (l+1)$ codimension one faces of $P$.

The barycentric subdivision $\beta P$ of $P$ is a simplicial complex. Observe that $P$ has at most $k+l+2$ codimension one faces, each of which has $(k-1)+l+2$ codimension one faces by above reasoning, and so on down to $k=1$ which has exactly 2 codimension one faces (the end points of the edge). So the number of $k$ dimensional simplexes of $\beta P$ is bounded by $(k+l+2)((k-1)+l+2)...(2+l+2)(2) = 2(k+l+2)!/(l+3)!$ by reasoning similar to that of Lemma \ref{barycentric}.

Note that strongly convex geometric triangulations are simplicial triangulations. Let $K_1 \cap K_2$ be a geometric polytopal complex of $M$ obtained by intersecting the geometric simplexes of $K_1$ and $K_2$. Observe that as the polytopes of $K_1 \cap K_2$ are obtained by the intersection of convex simplexes so they are convex in $M$ and their barycentric subdivision $K'=\beta(K_1 \cap K_2)$ is a geometric simplicial complex which is a common geometric subdivision of both $K_1$ and $K_2$. 

Let $s_i$ be the number of $i$-dimensional simplexes of $K'$ that lie in $K_1$. As each $i$-polytope $P$ of $K_1 \cap K_2$ that lies in the $i$-skeleton of $K_1$ is the intersection of a $i$-simplex of $K_1$ with some $j$ simplex of $K_2$ for $j \geq i$, so by above arguments its barycentric subdivision $\beta P$ has $2(i+j+2)!/(j+3)!$ many $i$-dimensional simplexes. As each simplex of $K_1$ and $K_2$ is strongly convex, their intersection is convex and hence connected. So there are at most $\sum_{j=i}^n p_i q_j$ many $i$-polytopes of $K_1 \cap K_2$ that lie in the $i$-skeleton of $K_1$. We therefore get $s_i \leq \sum_{j=i}^n \frac{2(i+j+2)!}{(j+3)!} p_i q_j$. 

Simplifying this by bounding $q_j$ with $\binom{n+1}{j+1}q_n$ and $(2n+2)!$ with $4(n+1)!^3$ gives us:
$$ \begin{array}{ccccc}
s_i & < & p_i \sum_{j=1}^n \frac{2(n+j+2)!}{(j+3)! (n-1)!} (n-1)! \binom{n+1}{j+1} q_n & < & 2(n-1)! \binom{2n+2}{n-1} (2^{n+1}-2) p_i q_n\\
 & < & \frac{4(2^{n}-1) (2n+2)!}{(n+3)!} p_i q_n & < & \frac{16}{(n+2)(n+3)} (2^n-1) (n+1)!^2 p_i q_n\\
 \end{array}
$$
As $n\geq 2$, we get the required bound.

\end{proof}

We now present some relations between the convexity radius and other invariants of the manifold.

\begin{definition}\cite{Dib}
Let $M$ be a Riemannian manifold and let $exp_p: T_p(M) \to M$ denote the exponential map at $p \in M$. The \emph{injectivity radius} at $p\in M$ is given by $inj(p)=\max\{R>0$ $|$ $exp_p |_{B(0,s)}$ is injective for all $0<s<R\}$, the \emph{convexity radius} at $p$ is given by $r(p) = \max\{ R>0$ $|$ $B(p, s)$ is strongly convex for all $0<s<R\}$ where $B(0,s) \subset T_p M$ denotes the Euclidean ball of radius $s$ around the origin and $B(p,s) \subset M$ denotes the ball of radius $s$ around $p$. The \emph{focal radius} at $p$ is defined as $r_f(p) = \min\{ T>0$ $|$ $\exists$ a non-trivial normal Jacobi field $J$ along a unit speed geodesic $\gamma$ with $\gamma(0) = p$, $J(0)=0$, and $||J||'(T)=0 \}$. If such a Jacobi field does not exist, then the focal radius is defined to be infinite. Globally, let $inj(M) = \inf_{p \in M} inj(p)$, $r(M) = \inf_{p\in M} r(p)$ and let $r_f(M)=\inf_{p\in M} r_f(p)$ respectively be the injectivity radius, convexity radius and focal radius of the manifold $M$.
\end{definition}

Applying the results of Dibble\cite{Dib} and Klingenberg\cite{Kli} to constant curvature manifolds, we get the following relation between convexity radius, injectivity radius and $l_c$, the length of smallest closed geodesic.
\begin{lemma}\label{inj}
For $M$ a spherical, Euclidean or hyperbolic closed manifold $$r(M) = \frac{1}{2} inj(M) = \frac{1}{4} l_c(M)$$
\end{lemma}
\begin{proof}
Theorem 2.6 of \cite{Dib} shows that when $M$ is compact, the convexity radius $r(M)$ equals $\min\{r_f(M), \frac{1}{4} l_c(M)\}$. When $M$ is hyperbolic or Euclidean, $r_f(M) = \infty$. When $M$ is spherical $r_f(M)=\pi/2$ and  $l_c(M)/4 \leq 2 diam(M)/4 \leq \pi/2$. So in either case, $r(M) = \frac{1}{4} l_c(M)$. Klingenberg\cite{Kli} has shown that $inj(M)=\min\{r_c(M), \frac{1}{2} l_c(M)\}$. For hyperbolic and Euclidean manifolds $r_c(M) = \infty$ and for spherical manifolds $r_c(M)=\pi$ and $\frac{1}{2} l_c(M) \leq \pi$, so in either case $inj(M)= \frac{1}{2} l_c(M)$.
\end{proof}

Cheeger's inequality roughly says that when we have an upper diameter bound, lower sectional curvature bound and lower volume bound we get a lower injectivity radius bound. The following is a sharper bound by  Heintze and Karcher (Corollary 2.3.2 of \cite{HeiKar}) which we state here only for constant curvature manifolds:

\begin{theorem}\cite{HeiKar}\label{HK}
Let $M$ be a complete spherical, Euclidean or hyperbolic $n$-manifold and let $\gamma$ be a closed geodesic in $M$. Then $l(\gamma) \geq 2\pi vol(M)/(\delta vol(\S^n))$ where $\S^n$ is the round $n$-sphere and
$$
\delta=\left\{
\begin{array}{cl}
diam(M)			&	\mbox{for $M$ Euclidean}\\
\sin^{n-1}(diam(M))	&	\mbox{for $M$ spherical}\\
\sinh^{n-1}(diam(M))	&	\mbox{for $M$ hyperbolic} \\
\end{array}\right.
$$
\end{theorem}

We use Lemma \ref{inj} and Theorem \ref{HK} to get a lower injectivity radius bound which is used to derive Corollary \ref{maincor} from Theorem \ref{mainthm}. In order to prove Theorem \ref{mainthm}, we first subdivide the given geometric triangulations sufficiently many times so that each simplex lies in a strongly convex ball. To bound the rate at which barycentric subdivisions scale the diameter of the simplex, we need the following theorem which we prove in Section \ref{scalingsec}.

\begin{theorem}\label{division}
Let $\beta^m \Delta$ be the $m$-th geometric barycentric subdivision of an $n$ simplex $\Delta$ with new vertices added at the centroid of simplexes. Let $\Lambda$ be an upper bound on the length of edges of $\Delta$. Then the diameter of simplexes of $\beta^m \Delta$ is at most $\kappa^m \Lambda$ where 
$$
\kappa=\left\{
\begin{array}{cl}
\frac{n}{n+1}			&	\mbox{for $M$ Euclidean}\\
\frac{2n}{2n+1}			&	\mbox{for $M$ spherical}\\
\frac{n\cosh^{n-1}(\Lambda)}{n\cosh^{n-1}(\Lambda) + 1}	&	\mbox{for $M$ hyperbolic} \\
\end{array}\right.
$$
\end{theorem}

We finally prove the main Theorem of this paper below.

\begin{proof}[Proof of Theorem \ref{mainthm}]
First assume that $K_1$ and $K_2$ are strongly convex geometric triangulations where links of all simplexes are shellable. By Lemma \ref{commonsub}, there exists a common geometric subdivision $K'$ of $K_1$ and $K_2$ with $s_i$ many $i$-simplexes in the $i$-skeleton of $K_1$. Using Lemma \ref{mainlemma} next we get a bound on the number of Pachner moves relating $\beta K_i$ and $\beta^2 K'$. Vertices that are common to both $K_1$ and $K_2$ are not removed by these Pachner moves.

Plugging in the bounds for $s_i$ from Lemma \ref{commonsub} in the formula obtained in Lemma \ref{mainlemma}, and then bounding $p_i$ by $\binom{n+1}{i+1} p_n$, we get the following bound for the number of moves relating $\beta K_1$ and $\beta^2 K'$. 

$$
\begin{array}{ll}
 	& \sum_{i=1}^n (n-i)! (i+1)!^2 p_{n-i-1} s_i \\
 < 	& (2^{n}-1) (n+1)!^2 q_n \sum_{i=1}^n (n-i)! (i+1)!^2 p_{n-i-1}p_i \\
 <  	& (2^n -1)(n+1)!^2 q_n \sum_{i=1}^n (n-i)! (i+1)!^2 \binom{n+1}{n-i} \binom{n+1}{i+1} p_n^2 \\
 < 	& (2^n - 1) (n+1)!^4 p_n^2 q_n \sum_{i=1}^n \frac{1}{(n-i)!}\\
 < 	& e \cdot (2^n-1)(n+1)!^4 p_n^2 q_n
\end{array}
$$

Exchanging the roles of $p_i$ and $q_i$ we get a bound on the number of moves relating $\beta K_2$ and $\beta^2 K'$. Summing them up we get the total number of moves needed to go from $\beta K_1$ to $\beta K_2$ as $e \cdot (2^n -1)(n+1)!^4 p_nq_n(p_n+q_n)$. To simplify notation we henceforth denote $p_n$ by $p$ and $q_n$ by $q$.

Given geometric triangulations $K_1$ and $K_2$ which may not be strongly convex, we need an integer $m$ such that $\beta^m K_i$ is strongly convex. That is, we need $m$ such that each simplex in $\beta^m K_i$ lies in a strongly convex ball or by Theorem \ref{division}, $\kappa^{m} \Lambda < 2r(M)$ where $\Lambda$ is an upper bound on the length of edges of $K_1$ and $K_2$. So we take $m$ to be any integer greater than $(\ln(2r(M))-\ln(\Lambda))/\ln(\kappa)$, as $\ln(\kappa)<0$. For $a>0$, $\ln(a+1) - \ln(a) =\int_{a}^{a+1} (1/x)dx > 1/(a+1)$. So we get $-1/\ln(\kappa) \leq \mu$ and we can take $m$ to be an integer greater than $\mu (\ln(\Lambda)-\ln(2r(M)))$ or by Lemma \ref{inj} we can take $m$ to be an integer greater than $\mu \ln(\Lambda/inj(M))$. For such values of $m$, $\beta^m K_1$ and $\beta^m K_2$ are strongly convex geometric triangulations.

When $n \leq 4$, the links of positive dimensional simplexes are  spheres of dimension at most $2$ and are therefore shellable. In general, to ensure that links of simplexes are shellable, by Theorem \ref{shellinglink}, we assume that $m$ is also greater than $2^{n+1}$. 

By Lemma \ref{barycentric}, the complexes $\beta^m K_1$ and $\beta^m K_2$ (with shellable links) have $(n+1)!^m p$ and $(n+1)!^m q$ many $n$-simplexes which are all strongly convex. So to go between $\beta \beta^{m} K_i$ we need $e \cdot (2^n -1) (n+1)!^{4+3m} p q (p + q)$ Pachner moves. By Lemma \ref{barymoves}, $\beta^{2^{n+1}} K_1$ and $\beta^{m+1} K_1$ are related by $(n+1)!^{2(m+1-2^{n+1}+1)}p^2 < (n+1)!^{2(m+2)} p^2$ moves (and similarly for $K_2$). So by the above arguments $\beta^{2^{n+1}} K_i$ are related by less than $e \cdot 2^n (n+1)!^{4+3m}pq(p + q)$ moves. When $n\leq 4$, the complexes $K_i$ are locally shellable. So by Lemma \ref{barymoves}, $K_1$ and $\beta^{m+1} K_1$ are related by $(n+1)!^{2(m+2)} p^2$ moves (and similarly for $K_2$).  Therefore we get the same bound $e \cdot 2^n (n+1)!^{4+3m}pq(p + q)$ on number of Pachner moves needed to go from $K_1$ to $K_2$ (instead of $\beta^{2^{n+1}} K_1$ to $\beta^{2^{n+1}} K_2$) when $n\leq 4$.

\end{proof}

\begin{proof}[Proof of Corollary \ref{isomcor}]
If $F:M \to N$ is an isometry then $F^{-1}(K_N)$ is a geometric triangulation of $M$. So by Theorem \ref{mainthm}, $K_1 = K_M$ and $K_2 = F^{-1}(K_N)$ are related by the given bounded number of Pachner moves. As $F$ is a simplicial isomorphism  from $K_2$ to $K_N$, we get the required result.

When $M$ and $N$ are complete finite volume hyperbolic manifolds of dimension at least 3, then by Mostow-Prasad \cite{Mos}\cite{Pra} rigidity, every homeomorphism is isotopic to an isometry. So if $K_M$ and $K_N$ are related by Pachner moves and simplicial isomorphisms, then $M$ and $N$ are homeomorphic and hence isometric. 

For dimensions up to 6, the PL and DIFF categories are isomorphic and by a theorem of de Rham \cite{DeR} diffeomorphic spherical manifolds are isometric, so the converse also holds for spherical manifolds of dimension at most 6. 

The converse is not true in the Euclidean case in any dimension as there are simplicially isomorphic flat tori which are not isometric. 
\end{proof}

\section{Subdivisions in constant curvature geometries}\label{scalingsec}
The aim of this section is to prove Theorem \ref{division} which gives the scaling factor for diameter of simplexes in the model geometries upon taking barycentric subdivisions.

\begin{definition}
Let $\Delta = [v_0, ..., v_n]$ be a geometric $n$-simplex. We define \emph{medians} and \emph{centroids} of faces of $\Delta$ inductively. Each vertex $v_i$ is defined to be its own centroid. We define the centroid of an edge of $\Delta$ as the midpoint of the edge. Having defined centroids of $k$ dimensional faces of $\Delta$, we define the medians of a $k+1$ dimensional face $\sigma$ as the geodesics in $\sigma$ joining a vertex of $\sigma$ to the centroid of its opposing $k$ dimensional face in $\sigma$. We define the centroid $c(\sigma)$ of $\sigma$ as the common intersection of all medians of $\sigma$. We shall show that such a common intersection exists for hyperbolic, spherical and Euclidean tetrahedra. Given simplexes $A$ and $B$ such that $\sigma = A\star B$, we define the medial segment joining $A$ and $B$ as the geodesic in $\sigma$ that connects the centroids $c(A)$ of $A$ and $c(B)$ of $B$. When $A$ or $B$ is a vertex the medial segment is a median.
\end{definition}

\begin{figure}
\centering
\includegraphics[width=0.4\textwidth]{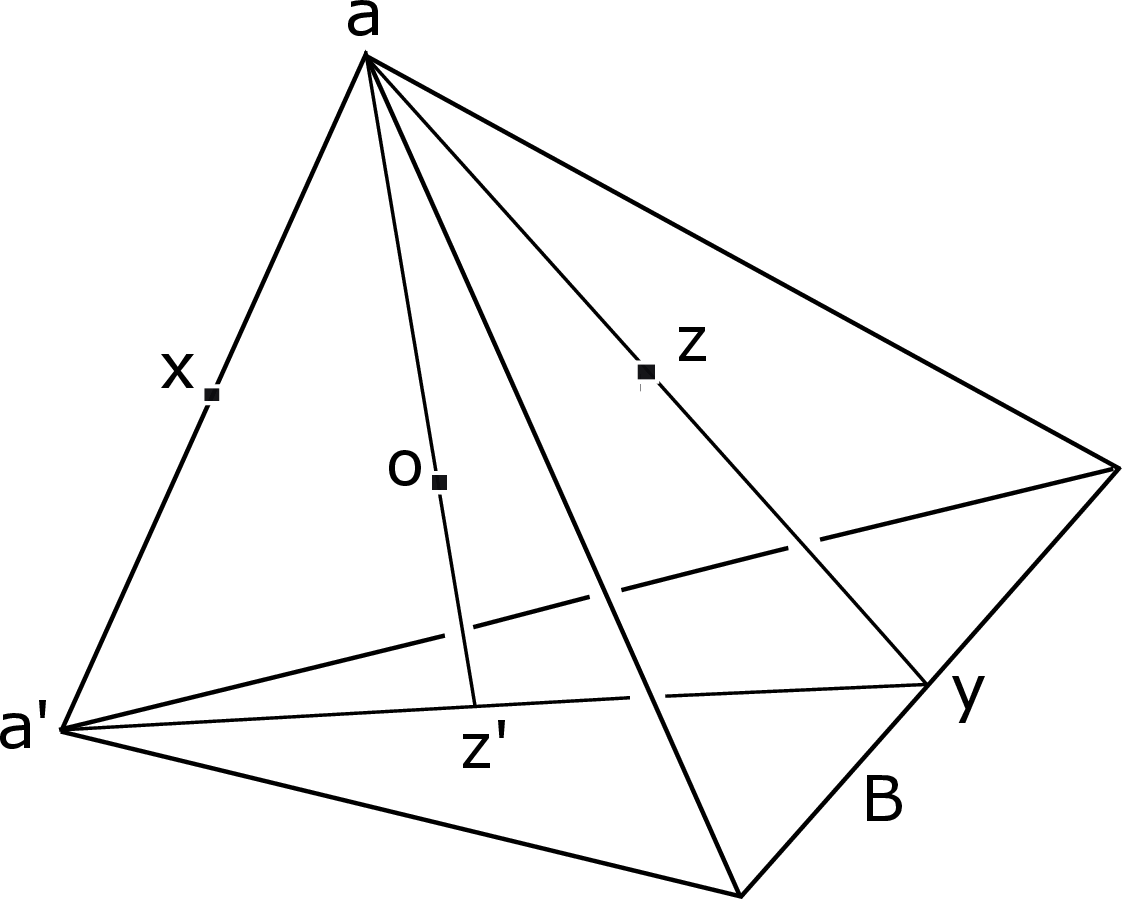}
\caption{An $n$-simplex $\Delta = a \star  a' \star  B$ with $x=c(a\star a')$, $y=c(B)$, $z=c(a\star B)$, $z'=c(a'\star B)$ and $o=c(\Delta)$, points on $\delta = [aa'y]$.}\label{ratiofig}
\end{figure}

\begin{lemma}\label{medianratio}
Let $\Delta$ be a Euclidean, hyperbolic or spherical $n$ dimensional simplex. All medial segments of $\Delta$ intersect at a common point $c(\Delta)$. Furthermore if $\Lambda$ is an upper bound for the length of the edges of $\Delta$ (with $\Lambda \leq \pi/2$ for $\Delta$ spherical) and $\Delta = a \star  B$ where $a$ is a vertex and $B$ is an $n-1$ dimensional face, then  $d(a, c(\Delta))/d(a, c(B)) \leq \kappa$ where $\kappa$ is as in Theorem \ref{division}.
\end{lemma}
\begin{proof}
\emph{Case I: $\Delta$ is Euclidean.}
Realise $\Delta$ as a linear combination of basis vectors $(v_i)$ in $\R^{n+1}$. For each face $\sigma = [v_{i_0}, ..., v_{i_k}]$ of $\Delta$, let $c(\sigma)=(v_{i_0}+...+v_{i_k})/(k+1)$. By induction on the dimension of $\sigma$, we shall show that $c(\sigma)$ is the centroid of $\sigma$.

When $\sigma$ is a vertex or an edge, $c(\sigma)$ is by definition the centroid of $\sigma$. Assume that the centroid is well defined for all faces of $\Delta$ of dimension less than $k$. After relabeling the vertices, assume that $\sigma = [v_0, ..., v_k]$ and $\sigma = A\star B$ with $A=[v_0,..., v_{p}]$ and $B=[v_{p+1}, ..., v_k]$. The dimensions of $A$ and $B$ are $p$ and $q=k-(p+1)$. We can express $c(\sigma)$ as a convex linear combination of $c(A)$ and $c(B)$ as below:
$$\begin{array}{lll}
c(\sigma) 	& = & \frac{\sum_{i=0}^k v_i}{k+1}\\
		& = & \frac{p+1}{k+1} \frac{\sum_{i=0}^p v_i}{p+1} + \frac{(k+1)-(p+1)}{k+1}  \frac{\sum_{i=p+1}^k v_i}{k-p}\\
		& = & \frac{p+1}{k+1} c(A) + \frac{q+1}{k+1} c(B)
\end{array}
$$

The point $c(\sigma)$ therefore lies on the medial segment connecting the centroids of $A$ and $B$. Furthermore it divides the medial segment $[c(A), c(B)]$ in the ratio $(q+1)/(p+1)$. Taking $\sigma = \Delta$ and $A$ as a vertex $a$ we get $d(a, c(\Delta))/d(c(\Delta), c(B))= n$, so that taking reciprocals and adding one on both sides gives $d(a,c(\Delta))/d(a, c(B))= n/n+1$ as required.\\

\emph{Case II: $\Delta$ is hyperbolic.}
Let $E^{(n,1)}$ be the $(n,1)$ Minkowski space, i.e. $\R^{n+1}$ with the inner product $u.v = u_1 v_1 +... +u_n v_n - u_{n+1} v_{n+1}$. The $n$ dimensional hyperbolic space $\H^n$ has a natural embedding in $E^{(n,1)}$ as the component of the hyperboloid $||x||^2=-1$ which lies in the upper half space of $\R^{n+1}$. Let $T=\{v \in \E^{(n,1)} : ||v|| <0\}$ and let (Euclidean) line segments in $T$ with endpoints on $\H^n$ be called the chords of $\H^n$. Let $p: T \to \H^n$ be the radial projection $x \to  \frac{\sqrt{-1}}{||x||} x$. It is easy to see that $p$ takes chords to hyperbolic geodesic segments in $\H^n$. To see that $p$ takes midpoints of chords to midpoints of the corresponding geodesic segment take $x$ and $y$ in $\H^n$ and let $r \in O^{+}(n, 1)$ restrict to an isometry of $\H^n$ that exchanges $x$ and $y$. Let $m=(x+y)/2$ be the midpoint of the chord joining $x$ and $y$ and let $z=p(m)$ be its image on the geodesic segment $[x,y]$. As there is a unique geodesic segment between pairs of points in $\H^n$, the isometry $r$ reflects the geodesic segment $[x, y]$ fixing only the mid point of $[x,y]$. But as $r$ is linear in $\R^{n+1}$,  $r(z) =\frac{\sqrt{-1}}{||m||} r(m) = \frac{\sqrt{-1}}{||m||}m = z$, so $z$ is the midpoint of $[x,y]$.
 
Given a hyperbolic simplex $\Delta$ in $\H^n$ with vertices $v_i$, let $\Delta_0$ be the Euclidean convex linear combination of $v_i$ in $\R^{n+1}$. As the homeomorphism $p|_{\Delta_0}: \Delta_0 \to \Delta$, fixes the vertices and takes midpoints of edges to midpoints of edges, by induction, it takes medial segments to medials segments and hence takes centroids to centroids. In particular, all the medial segments of $\Delta$ intersect at the common point $c(\Delta)$ as in the Euclidean case.

For points $a$, $x$, $b$ in $\Delta$, define the ratio $h(a, x, b) = \sinh(d(a, x))/\sinh(d(x, b))$.  By induction on the dimension of $\Delta$ we shall prove that if $\Delta = a\star B$ with $a$ a vertex and $B$ an $n-1$ face, then $1 \leq h(a, c(\Delta), c(B)) \leq n \cosh^{n-1}(\Lambda)$. When $\Delta = a\star b$ is an edge, then $h(a, c(\Delta), b) = 1$. Let $\Delta = a \star  a' \star  B$ be an $n$ dimensional simplex. Let $\delta$ be the geodesic triangle $[a, a', c(B)]$ in $\Delta$. Let $x=c(a\star a')$, $y=c(B)$, $z=c(a\star B)$, $z'=c(a'\star B)$ and $o=c(\Delta)$ be points of $\delta$ as in Figure \ref{ratiofig}. As the medial segments of $\Delta$ all intersect at the centroid $o$, the segments $[a, z']$, $[a', z]$ and $[x, y]$ of $\delta$ have a common intersection  at $o$.  By the hyperbolic version of van Obel's Theorem, 
$$h(a, o, z')= \cosh(d(a', z')) h(a, x, a') + \cosh(d(z', y)) h(a, z, y)$$
As $x$ is the midpoint of $[a, a']$ so $h(a, x, a')=1$ and by induction applied to the $n-1$ simplex $a\star B$, $1 \leq h(a, z, y)=h(a, c(a\star B), c(B)) \leq (n-1) \cosh^{n-2} (\Lambda)$. As $1 \leq \cosh$,  $1 \leq h(a, o, z') \leq n \cosh^{n-1}(\Lambda)$ as required.

Define $f(x) = \sinh(x)/x$ for $x>0$ and $f(0)=1$. Then $f'(x) = (x \cosh(x) - \sinh(x))/x^2$ has positive numerator because it takes value $0$ at $0$ and it's derivative is positive. So $f$ is an increasing function. For $0<x \leq y$, $\sinh(x)/x \leq \sinh(y)/y$, i.e, $y/x \leq \sinh(y)/\sinh(x)$. As $h(a, o, z') \geq 1$, so $\sinh(d(a,o)) \geq \sinh(d(o,z'))$ and as $\sinh$ is a strictly increasing function so $d(a,o) \geq d(o,z')$. By above arguments then $d(a,o)/d(o,z') \leq h(a, o, z') \leq n \cosh^{n-1}(\Lambda)$. Taking reciprocals and adding one on both sides we get the required bound $\kappa$.\\

\emph{Case III: $\Delta$ is spherical.} Taking the standard embedding of $\S^n$ in $\R^{n+1}$ with $p: \R^{n+1} \setminus 0 \to \S^n$ as the radial projection $p(x) = \frac{x}{||x||}$ we can show that medial segments of a spherical simplex $\Delta$ have a common intersection at the centroid, as in the hyperbolic case.

Proceeding as in the hyperbolic case, using $s(a, x, b) = \sin(d(a, x))/\sin(d(x, b))$ instead of $h(a, x, b)$ and using the spherical van Obel theorem $$s(a, o, z')= \cos(d(a', z')) s(a, x, a') + \cos(d(z', y)) s(a, z, y)$$ we get the bound $s(a, o, z') \leq n$.

Suppose that for $0<p, q \leq \pi/2$, we are given $\sin(p)/\sin(q) \leq n$. Then we shall show that $p/q \leq 2n$. As $\sin(q) \leq q$ for $q>0$, so $\sin(p)/q \leq \sin(p)/\sin(q) \leq n$. Let $0<t_0< \pi/2$ be the point where $\sin(t_0) = \pi/4$. When $ t_0 \leq p\leq \pi/2$, $\sin(t_0) \leq \sin(p)$ so $\sin(t_0)/q \leq \sin(p)/q \leq n$ and we get $p/q \leq n \pi/(2\sin(t_0)) =2n$. When $0 < p \leq t_0$, $\cos(t_0) \leq \cos(p)$ and as $p \leq \tan(p)$ (see the power series expansion of tan for this relation) so $p \cos(p)/q \leq \sin(p)/q \leq n$. We therefore get $p/q \leq n/\cos(t_0) \leq 2n$ as $\cos(t_0) \geq 1/2$. Taken together we conclude that $p/q \leq 2n$ as required. As $s(a,o,z') \leq n$, $d(a,o)/d(o,z') \leq 2n$ and adding one and taking reciprocals gives the required bound $\kappa$ in the spherical case.


\end{proof}

\begin{lemma}\label{adjedges}
Let $ABC$ be a hyperbolic, Euclidean or spherical triangle. When $ABC$ is spherical we assume that the length of edges of $ABC$ is at most $\pi/2$. Then for any point $D$ on the segment $[B, C]$, $d(A, D) \leq \max(d(A, B), d(A, C))$.
\end{lemma}
\begin{proof}
Suppose that $ABC$ is a hyperbolic or Euclidean triangle for which the lemma is not true. Then the angle $ADB$ is less than angle $B$ and angle $ADC$ is less than angle $C$ which would imply that the sum of angles $B$ and $C$ is greater than $\pi$, a contradiction. 

Let $ABC$ be a spherical isosceles triangle in $S^2 \subset \R^3$ with $A$ at the north pole and with base $BC$ having $z$ coordinate $z_0 \geq 0$. The plane containing the origin, $B$ and $C$ intersects $S^2$ in the spherical geodesic segment $[B,C]$ which lies in the half space $z \geq z_0$. So for any point $D \in [B, C]$, $d(A, D) \leq d(A, B)$. When $ABC$ is an arbitrary spherical triangle with $A$ at the north pole, side $AB$ longer than side $AC$ and $z_0$ as the $z$-coordinate of $B$, we extend the side $AC$ to the point $C'$ which has $z$ coordinate $z_0$ so that $ABC'$ is an isosceles triangle. For any point $D \in [B, C]$, extend the segment $[A, D]$ to $D' \in [B C']$, then by the above argument $d(A, D) \leq d(A, D') \leq d(A, B)$.
\end{proof}

\begin{lemma}\label{diam}
Let $\Delta$ be a hyperbolic, spherical or Euclidean simplex. If $\Delta$ is spherical we assume the length of its edges is at most $\pi/2$. Then the diameter of $\Delta$ is the length of the longest edge of $\Delta$.
\end{lemma}
\begin{proof}
Let $[x,y]$ be a maximal segment in $\Delta$ and assume that it does not lie in any proper simplex of $\Delta$. Let $x \in A$, $y \in B$ for simplexes $A$ and $B$ in $\del \Delta$ then $\Delta = A\star B$. If both $x$ and $y$ are vertices then trivially, $d(x,y)=l([x,y])$ is at most length of longest edge of $\Delta$. If $x$ is not a vertex, then let $A=a\star A'$ with $a$ a vertex of $A$. Extend the segment $[a, x]$ to $x' \in A'$. Applying Lemma \ref{adjedges} to the triangle $[ax'y]$, $d(y, x) \leq \max(d(y, a), d(y, x'))$. As dimensions of $a\star B$ and $A'\star B$ are both less than dimension of $\Delta$, so by induction $d(y, x)$ is at most the length of the longest edge of $\Delta$.
\end{proof}

Note that Lemma \ref{diam} is not true for spherical triangles with edges longer than $\pi/2$ as can be seen by taking an isosceles triangle with base length less than $\pi/2$ and the equal length edges of length more than $\pi/2$. The diameter of such a triangle is the length of the altitude on the base, which is greater than the length of all the edges.

We are finally in a position to prove the main Theorem of this section:
\begin{proof}[Proof of Theorem \ref{division}]
We shall first show, by induction on the dimension of faces $A$ of $\Delta$, that $d(c(A), c(\Delta)) \leq \kappa \Lambda$. When $A$ is a vertex, by Lemma \ref{medianratio} and Lemma \ref{diam}, $d(a, c(\Delta)) \leq \kappa d(a, c(B)) \leq \kappa \diam(\Delta) \leq \kappa \Lambda$. For $A=a\star A'$, consider the triangle $T = [a, c(A'), c(\Delta)]$. As the medial segment $[a, c(A')]$ passes through $c(A)$, the segment $[c(\Delta), c(A)]$ lies in $T$ and by Lemma \ref{adjedges}, $d(c(\Delta), c(A))$ is at most $\max(d(c(\Delta), a), d(c(\Delta), c(A')))$ which is in turn bounded by $\kappa \Lambda$ by induction.

Each edge of $\beta \Delta$ is a medial segment in some simplex $\delta \in \Delta$, of the kind $[c(\delta), c(A)]$ for $A \in \delta$. By above arguments, length of such edges is bounded by $\kappa \Lambda$. Repeating the argument for $\beta \Delta$ in place of $\Delta$, taking $\kappa \Lambda$ as the upper bound for length of edges, we get the bound $\kappa^2 \Lambda$ for edges of $\beta^2 \Delta$. Repeating the argument $m$ times and applying Lemma \ref{diam}, we get the required upper bound for the diameter of simplexes of $\beta \Delta$.
\end{proof}

To see that the constant $\kappa$ in the hyperbolic case can not be made independent of the length of the edges, consider a hyperbolic isosceles triangle $\Delta = ABC$ with base $BC$. Let $a$ and $b$ be the length of the sides opposite to vertices $A$ and $B$, let $m$ be the length of the median from $A$ and let $x$ be the distance from $A$ to the centroid of $ABC$. Assume that $m=ya$ for some $y>0$. By the hyperbolic version of Pythagorean theorem, $\cosh(b)=\cosh(a/2) \cosh(m)$ which gives the following for all $a>0$:
$$ 1 \leq b/m  = 	 \frac{\cosh^{-1}(\cosh(a/2)\cosh(ya))}{ya}  \leq  \frac{\cosh^{-1}(\cosh(ya + a/2))}{ya} = 	 1+ \frac{1}{2y}$$

So for any fixed base length $a$ and isosceles triangle as above with $m=ya$, $lim_{y\to \infty} m/b \to 1$. Also, as $\sinh(x)/\sinh(m) = 2\cosh(a/2)/(2\cosh(a/2)+1) \to 1$ as $a\to \infty$. So for large enough $a$ and $y$, $x/b = (x/m) (m/b)$ is as close to $1$ as required. In other words, the diameter of simplexes in $\beta \Delta$ can be made arbitrarily close to the diameter of $\Delta$.

\begin{Data}
Data sharing not applicable to this article as no datasets were generated or analysed during the current study.
\end{Data}

\begin{acknowledgements}
The first author was supported by the MATRICS grant of Science and Engineering Research Board, GoI and the second author was supported by an award from the National Board of Higher Mathematics, GoI.
\end{acknowledgements}

\bibliographystyle{alpha}
\bibliography{Final}

\end{document}